\documentclass[reqno]{amsart}
\usepackage[english]{babel}
\usepackage{amssymb,amsmath,verbatim,xcolor}
\usepackage[T1]{fontenc}
\newtheorem{theorem}{Theorem}[section]

\theoremstyle{definition}
\newtheorem{definition}[theorem]{Definition}
\newtheorem{example}[theorem]{Example}
\theoremstyle{remark}
\newtheorem*{remark}{Remark}
\numberwithin{equation}{section}

\DeclareMathOperator{\EntOp}{\mathbf{Ent}}
\newcommand{\Ent}{\EntOp}

\newcommand{\CEP}[1]{\mbox{$\mathbb{C}^{#1}$}}
\newcommand{\C}{\mbox{$\mathbb{C}$}}

\newcommand{\Eo}{\mbox{$\mathcal{E}_{0}$}}
\newcommand{\Ep}{\mbox{$\mathcal{E}_{p}$}}

\newcommand{\Jp}{\operatorname{J}_p}

\newcommand{\ddcn}[1]{\mbox{$\left(dd^{c}#1\right)^{n}$}}

\begin{document}

\title{Generalization of Finite Entropy Measures in Kähler Geometry}

\author{Per \AA hag}\address{Department of Mathematics and Mathematical Statistics\\ Ume\aa \ University\\SE-901~87 Ume\aa \\ Sweden}\email{per.ahag@umu.se}

\author{Rafa\l\ Czy{\.z}}\address{Faculty of Mathematics and Computer Science, Jagiellonian University, \L ojasiewicza~6, 30-348 Krak\'ow, Poland}
\email{rafal.czyz@im.uj.edu.pl}


\begin{abstract}
In this paper, we extend the concept of finite entropy measures in Kähler geometry. We define the finite $p$-entropy related to $\omega$-plurisubharmonic functions and demonstrate their inclusion in an appropriate energy class. Our study is anchored in the analysis of finite entropy measures on compact Kähler manifolds, drawing inspiration from fundamental works of Di Nezza, Guedj, and Lu. Utilizing a celebrated result by Darvas on the existence of a Finsler metric on the energy classes, we conclude this paper with a stability result for the complex Monge-Ampère equation.
\end{abstract}

\maketitle

\bigskip

\section{Introduction}

\bigskip

The concept of entropy, fundamental to understanding uncertainty in probability distributions, is defined for a discrete measure $P$ over outcomes $\{x_1, x_2, \ldots, x_n\}$ as
\[
H(P) = -\sum_{i=1}^{n} P(x_i) \log_b P(x_i),
\]
where $P(x_i)$ is the probability of outcome $x_i$. This measure is crucial in quantifying a system's randomness or unpredictability and broad applications across data compression, cryptography, machine learning, statistical mechanics, and information theory. Notably, the finite entropy of probability measures also finds significant application in complex geometry, influencing the search  of certain metrics.

In this paper, we draw inspiration from the work of Di Nezza, Guedj, and Lu~\cite{DNGL}, who have established significant results concerning compact K\"ahler manifolds $(X, \omega)$. In their study, $\omega$ represents a normalized K\"ahler form on $X$, which has complex dimension $n \geq 2$. They focused on analyzing the probability measure $\mu = (dd^c u + \omega)^n = f \omega^n$, having finite entropy, as defined by the integral:
\begin{equation}\label{Intr_entr}
\int_X f(\log f) \omega^n < \infty.
\end{equation}
In this context, they proved that $u \in \mathcal{E}_{\frac{n}{n-1}}$, referring to Cegrell's class $\mathcal{E}_{\frac{n}{n-1}}$ of $\omega$-plurisubharmonic functions with $\frac{n}{n-1}$--energy (see~\eqref{Def Ep} for the definition). Their inspiration was drawn from the result~\cite[Theorem 2.17]{BBEGZ}, which proved that all $\omega$-plurisubharmonic functions with finite entropy belong to the Cegrell class $\mathcal{E}_1$.

\medskip

Building on this foundation, we aim to prove the following generalization.

\bigskip

\noindent {\bf Theorem~A.} \emph{Let $u$ be an $\omega$-plurisubharmonic function
on a connected and compact K\"ahler manifold $(X,\omega)$, with $\omega$ being  a normalized K\"ahler form on $X$ of complex dimension $n \geq 2$, and $\sup_X u=-1$. Let $0<p<n$. If the probability measure $\mu = (dd^c u + \omega)^n = f\omega^n$ has finite $p$-entropy, i.e.,
\begin{equation}\label{Pentr_I}
\int_X (1+f)(\log (1+f))^p \omega^n < \infty,
\end{equation}
then $u$ belongs to $\mathcal{E}_{\frac{np}{n-p}}$. If additionally,
\begin{equation}\label{Pentr_II}
\int_X (1+f)(\log (1+f))^p \omega^n < A,
\end{equation}
for a constant $A > 0$, then there are constants $c, C > 0$ depending only on $n, p, A$ such that any probability measure $\mu = (dd^c u + \omega)^n = f \omega^n$ with uniform bound of its $p$-entropy in the sense of~\eqref{Pentr_II} satisfies:
\[
\int_X \exp\left(c(n,p,A)(-u)^{\frac{n}{n-p}}\right) \omega^n < C(n,p,A)\quad \text{and} \quad e_{\frac{np}{n-p}}(u) < C(n,p,A).
\]}

\medskip

Directly mimicking~\eqref{Intr_entr} to define the $p$-entropy is not feasible, as we need to circumvent situations where $\log f < 0$ to enable taking the power $p$ for any $p > 0$. Consequently, we adopt condition~\eqref{Pentr_I} in place of~\eqref{Intr_entr}. It is crucial to note that the space of probability measures with finite entropy remains consistent, whether we use condition~\eqref{Intr_entr}, or~\eqref{Pentr_I} with $p=1$.

\bigskip

Our main work of proving Theorem~A will be conducted in the local context in a bounded hyperconvex domain in $\mathbb{C}^n$, $n\geq 2$, and our methods use an interplay of the Cegrell techniques in pluripotential theory and Orlicz spaces from functional analysis. Our main work to prove Theorem~A is the following local version (Theorem~B). Fix $p>0$. We say that a plurisubharmonic function $u$ belongs to $\Ent_p$ if $u$ is in the Cegrell class $\mathcal F$ with the measure $(dd^cu)^n=fdV_{2n}$ having
finite $p$-entropy
\[
\int_{\Omega}(1+f)(\log (1+f))^p dV_{2n}<\infty.
\]
Moreover for $A>0$ by $\Ent_p^A$ denote the class of $u \in \Ent_p$ such that
\[
\int_{\Omega}(1+f)(\log (1+f))^p dV_{2n}\leq A.
\]
Proceeding, we present Theorem~B.

\bigskip

\noindent {\bf Theorem~B.} \emph{Let $\Omega$ be a bounded hyperconvex domain in  $\mathbb{C}^n$, $n\geq 2$ and let $0 < p < n$. Then
\[
\Ent_p \subset \mathcal{E}_{\frac{np}{n-p}}.
\]
Furthermore, for $A > 0$, there exist constants $c, C > 0$ depending only on $n, p, A$ such that for any $u \in \Ent_p^A$,
\[
\int_{\Omega} \exp\left(c(n,p,A)(-u)^{\frac{n}{n-p}}\right) dV_{2n} < C(n,p,A).
 \quad \text{and} \quad
 e_{\frac{np}{n-p}}(u) < C(n,p,A)
\]
For $p \geq n$, it holds that $\Ent_p \subset \mathcal{E}_{q}$ for all $q > 0$.
}

\medskip

The case $p>n$ in Theorem~B is due to Ko\l odziej~\cite{Ko96}. Finally, we would like to mention that using non-pluripotential methods, Wang, Wang, and Zhou in~\cite{WWZ} proved the analog of Theorem~B in the special case of smooth functions.

\bigskip

This note is organized as follows: Section~\ref{Sec Back} provides background on Orlicz Spaces and the Cegrell Classes. In Section~\ref{ThmAB}, we present the proofs of Theorems A and B. Finally, we conclude this note in Section~\ref{Sec_Stability} by applying Theorem~A and Theorem~B to prove stability results for the solution to the complex Monge-Ampère equation both on hyperconvex domains in $\C^n$ (Theorem~\ref{thm_stability}) and on compact Kähler manifolds (Theorem~\ref{thm_stability_cKm}).

\section{Background on Orlicz Spaces and the Cegrell Classes}\label{Sec Back}
This section establishes the foundational concepts of Orlicz Spaces and the Cegrell Classes. Orlicz Spaces, extending classical $L^p$ spaces, facilitate the analysis of functions with variable growth rates. The development of these spaces is well-documented in~\cite{M89}. Similarly, the Cegrell Classes, introduced in~\cite{C98,C04,C08}, provide a function spaces approach to studying the complex Monge-Ampère equation and its application. Additional insights into pluripotential theory can be found in~\cite{Czy09,GZbook,KlimekBook}. This section reviews these essential concepts, serving as a primer for their application and interrelation in subsequent sections.

Consider an increasing, continuous, convex function $\varphi: [0, \infty) \to [0, \infty)$ satisfying the following properties:
\[
\varphi(0) = 0, \quad\lim_{t \to 0^+} \frac{\varphi(t)}{t} = 0,\; \text{and}\quad\lim_{t \to \infty} \frac{\varphi(t)}{t} = \infty.
\]
Let $(T,\frak M,\mu)$ be a measurable space, where $T$ is an abstract space, $\frak M$ is a $\sigma$-algebra and $\mu$ is a $\sigma$-finite, positive, complete and atomless measure. Later in the paper we shall consider two cases: the first when $T$ is a bounded domain $\Omega$ in $\mathbb C^n$ with Lebesgue measure $dV_{2n}$, and the second when $T$ is compact K\"ahler manifold $(X,\omega)$ with the volume form $\omega^n$.

Let $Y$ denote the space of measurable functions with respect to the measure $\mu$ on $T$. The modular function $\rho$ on $Y$ is defined by
\[
\rho(f) = \int_{T} \varphi(|f|) d\mu.
\]
Define the Orlicz class $L_0^{\varphi}$ as
\[
L_0^{\varphi} = \left\{ f \in Y: \rho(f) < \infty \right\},
\]
and the Orlicz class $L_a^{\varphi}$ as
\[
L_a^{\varphi} = \left\{ f \in Y : \rho(\lambda f) < \infty \text{ for every } \lambda > 0 \right\}.
\]
The Orlicz space $L^{\varphi}$ is the smallest vector space containing $L_0^{\varphi}$. Moreover, $L_a^{\varphi}$ is the largest vector space containing $L_0^{\varphi}$ that is closed in $L^{\varphi}$.

Convergence in the norm in $L^{\varphi}$, indicated by $f_j \to f$ as $j \to \infty$, means that for all $\lambda > 0$,
\[
\int_{T} \varphi(\lambda|f_j - f|) d\mu \to 0 \quad \text{ as } j \to \infty.
\]
Let us recall also so called modular convergence, which means that there exist a constant $\lambda > 0$,
for which
\[
\int_{T} \varphi(\lambda|f_j - f|) d\mu \to 0 \quad \text{ as } j \to \infty.
\]
It was proved in~\cite{MO59} that, under the assumption that $\varphi$ satisfies so called condition $\Delta_2$ that there is a constant $C > 0$ such that $\varphi(2t) \leq C\varphi(t)$ for all $t$, norm convergence, and modular convergence are equivalent.

For increasing, continuous, unbounded functions $\alpha$ and $\beta$ with $\alpha(0) = \beta(0)$, we have:
\begin{enumerate}
\item $L_0^{\alpha} \subset L_0^{\beta}$ if and only if $\limsup_{t \to \infty} \frac{\beta(t)}{\alpha(t)} < \infty$;
\item $L^{\alpha} \subset L^{\beta}$ if and only if there exists some $\lambda > 0$ such that $\limsup_{t \to \infty} \frac{\beta(\lambda t)}{\alpha(t)} < \infty$.
\end{enumerate}

$L^{\varphi}$ is a Fréchet space. If $\varphi$ is convex, then $L^{\varphi}$ is a Banach space with the Orlicz norm
\[
\|f\|_{\varphi}^0 = \sup \left\{ \int_{T} |fg| d\mu : \int_{T} \varphi(|g|) d\mu \leq 1 \right\}
\]
or the equivalent Luxemburg norm
\[
\|f\|_{\varphi} = \inf \left\{ \lambda > 0 : \int_{T} \varphi\left(\frac{|f|}{\lambda}\right) d\mu\leq 1 \right\}.
\]
Given $\varphi$ as above, we can represent it as
\[
\varphi(t) = \int_0^t p(s) \, ds,
\]
where $p$ is a function related to $\varphi$. The right-inverse of $p$, denoted by $q(t)$, is defined as $q(t) = \sup\{s : p(s) \leq t\}$. The complementary function, $\varphi^*(s)$, important in Orlicz space theory, is then given by
\[
\varphi^*(s) = \int_0^s q(t) \, dt.
\]
This integral represents $\varphi^*(s)$, and it is essential in defining the dual space in Orlicz space theory, fulfilling the role of the Legendre transform of $\varphi$ in this context. We introduce Young's inequality as follows.

\begin{theorem}\label{yi}
Under the assumptions that $\varphi: [0, \infty) \to [0, \infty)$ is increasing, continuous, and satisfies $\varphi(0) = 0$ and $\lim_{t \to \infty} \frac{\varphi(t)}{t} = \infty$, for all $t, s \geq 0$, we have
\[
st \leq \varphi(t) + \varphi^*(s).
\]
Furthermore,
\begin{equation}\label{star}
\varphi^*(s) = \sup_{t \geq 0} (st - \varphi(t)), \quad \text{and} \quad \varphi(t) = \sup_{s \geq 0} (st - \varphi^*(s)).
\end{equation}
\end{theorem}

Next, we present a counterpart of the H\"older inequality.

\begin{theorem}\label{orlicz_holder}
For $f \in L^{\varphi}$ and $g \in L^{\varphi^*}$, it holds that
\[
\left| \int_{T} f g \, d\mu \right| \leq \|f\|^0_{\varphi} \|g\|_{\varphi^*},
\]
and
\[
\left| \int_{T} f g \, d\mu \right| \leq \|f\|_{\varphi} \|g\|^0_{\varphi^*}.
\]
\end{theorem}

We are interested in the case where the Orlicz space corresponds to the $p$-entropy class. Recall that in dealing with the entropy class for $\varphi_1(t) = (1+t)\log(1+t)$, then $\varphi_1^*(s)$ is essentially equal to $e^s$. We aim to generalize this to the $p$-entropy class. We will need the following example in the proof of Theorem~A and~B.

\begin{example}\label{ex}
Consider the function defined by
\[
\Phi_p(t) = \varphi_p(t) - \psi_p(t) = (1 + t)(\log(1 + t))^p - p \int_0^t (\log(1 + x))^{p-1} \, dx.
\]
Note that $\Phi_p$ is a continuous, increasing, and convex function with the properties that $\Phi_p(0) = 0$, $\lim_{t \to \infty} \frac{\Phi_p(t)}{t} = \infty$, $\lim_{t \to 0^+} \frac{\Phi_p(t)}{t} = 0$, and
\[
\lim_{t \to \infty} \frac{\Phi_p(t)}{\varphi_p(t)} = 1.
\]
Consequently, we have $L_0^{\varphi_p} = L_0^{\Phi_p}$, and  $L^{\varphi_p} = L^{\Phi_p}$.

Note that $\varphi_p(2t)\leq 2^{p+1}\varphi_p(t)$ so it satisfies $\Delta_2$ condition. Moreover since
\[
\lim_{t \to 0} \frac{\Phi_p(2t)}{\Phi_p(t)} = 2^{p+1}\ \text{and} \ \lim_{t \to \infty} \frac{\Phi_p(2t)}{\Phi_p(t)} = 2,
\]
then $\Phi_p$ satisfies $\Delta_2$ condition too. Finally let us mention that as a consequence we get
$L^{\Phi_p}=L_0^{\Phi_p}=L_a^{\Phi_p}$.

 We shall now consider the Legendre transform $\Phi_p^*$ of $\Phi_p$. The supremum of the function $t\mapsto st - \Phi_p(t)$ is achieved at $t = \exp\left(s^{\frac{1}{p}}\right) - 1$, thus
\[
\Phi^*_p(s) = p \int_0^{\exp(s^{\frac{1}{p}}) - 1} (\log(1 + x))^{p-1} \, dx - s.
\]
Further, considering the derivative,
\[
(\Phi^*_p(s))' = \exp(s^{\frac{1}{p}}) - 1,
\]
and applying the Lagrange mean value theorem, we obtain
\begin{equation}\label{e0}
\begin{aligned}
\Phi^*_p(s)  & = \Big|\Phi^*_p(s) - \Phi^*_p(0)\Big|
 \leq \sup_{t \in [0,s]} \left|\exp(t^{\frac{1}{p}}) - 1\right|s
 \leq \left(\exp(s^{\frac{1}{p}}) - 1\right)s \\
 & \leq \exp((\frac pe + 1)s^{\frac{1}{p}})
  = \exp\left(c_p s^{\frac{1}{p}}\right),
 \end{aligned}
\end{equation}
where $c_p = \frac pe + 1$. \hfill{$\Box$}
\end{example}

With the groundwork laid in Orlicz spaces, we next present the necessary definitions and facts of the  Cegrell classes.

\begin{definition}\label{prel_hcx}  A set $\Omega\subseteq\CEP{n}$, $n\geq 1$, is a bounded hyperconvex domain if it is a bounded, connected, and open set such
 that there exists a bounded plurisubharmonic function $u:\Omega\rightarrow (-\infty,0)$ such that the closure of the
 set $\{z\in\Omega : u(z)<c\}$ is compact in $\Omega$, for every $c\in (-\infty, 0)$.
\end{definition}

\begin{definition}\label{ep}
We say that a plurisubharmonic function $u$ defined on $\Omega$ belongs to $\Eo$ $(=\Eo
(\Omega))$ if it is bounded function,
\[
\lim_{z\rightarrow\xi} u (z)=0 \quad \text{ for every } \xi\in\partial\Omega\, ,
\]
and
\[
\int_{\Omega} \ddcn{u}<\infty\, ,
\]
where $\ddcn{\cdot}$ is the complex Monge-Amp\`{e}re operator. Furthermore, we say that $u\in \Ep$ $(=\Ep(\Omega))$, $p>0$, if $u$ is a plurisubharmonic function
defined on $\Omega$ and there exists a decreasing sequence, $\{u_{j}\}$, $u_{j}\in\Eo$, that converges pointwise to $u$ on $\Omega$,
as $j$ tends to $\infty$, and
\[
\sup_{j} e_p(u_j)=\sup_{j}\int_{\Omega}
(-u_{j})^{p}\ddcn{u_{j}}< \infty\, .
\]
Finally, we say that $u\in \mathcal F$ $(=\mathcal F(\Omega))$, if $u$ is a plurisubharmonic function
defined on $\Omega$ and there exists a decreasing sequence, $\{u_{j}\}$, $u_{j}\in\Eo$, that converges pointwise to $u$ on $\Omega$,
as $j$ tends to $\infty$, and
\[
\sup_{j}\int_{\Omega}\ddcn{u_{j}}< \infty\, .
\]
\end{definition}

We shall need on the following inequality. This theorem  was proved in~\cite{P99} for $p\geq 1$, and for $0<p<1$
in~\cite{ACP07} (see also~\cite{C98,CP97}).

\begin{theorem}\label{thm_holder} For $p > 0$ and functions $u_0, u_1, \ldots, u_n \in \mathcal{E}_p$ with $n \geq 2$, there exists a constant $d(p, n, \Omega) \geq 1$ dependent only on $p$, $n$, and $\Omega$, such that
\begin{multline*}
\int_\Omega (-u_0)^p dd^c u_1\wedge\cdots\wedge dd^c u_n\\ \leq
d(p,n,\Omega)\; e_p(u_0)^{p/(p+n)}e_p(u_1)^{1/(p+n)}\cdots
e_p(u_n)^{1/(p+n)}\, .
\end{multline*}
\end{theorem}

\section{Proof of Theorem~A and Theorem~B}\label{ThmAB}

In this section, we will present proofs of Theorem~A and Theorem~B, starting with the latter.

Let $\Omega$ be a bounded hyperconvex domain in $\mathbb{C}^n$. For $p>0$, define
\[
\varphi_p(t)=(1+t)(\log(1+t))^p,
\]
and note that if $f \in L^{\varphi_p}$, then
\[
\begin{aligned}
&\int_{\Omega}f dV_{2n} = \int_{\{z \in \Omega : f(z) \leq e\}}f dV_{2n} + \int_{\{z \in \Omega : f(z) > e\}}f dV_{2n} \\
&\leq eV_{2n}(\Omega) + \int_{\Omega}(1+f)(\log (1+f))^p dV_{2n} < \infty.
\end{aligned}
\]
Thus, $L^{\varphi_p} \subset L^1$, which means that the following definition is well-posed.

\begin{definition}
 We say that a plurisubharmonic function $u$ belongs to $\Ent_p$ if $u\in\mathcal{F}$ and $(dd^c u)^n = f dV_{2n}$ with
\begin{equation}\label{CondP}
\int_{\Omega} (1+f)(\log(1+f))^p dV_{2n} < \infty.
\end{equation}
Moreover, for $A > 0$, denote by $\Ent_p^A$ the subset of $\Ent_p$ with the additional requirement:
\[
\int_{\Omega} (1+f)(\log (1+f))^p \, dV_{2n} \leq A.
\]
\end{definition}
It is noteworthy that condition~\eqref{CondP} is equivalent to
\[
\int_{\Omega} f(|\log f|)^p dV_{2n} < \infty,
\]
however, we have chosen~\eqref{CondP} as it proves more convenient for our purposes, as demonstrated in Example~\ref{ex}. For information about the entropy in the relative setting, see e.g.~\cite{DNTT23}.

We shall need the following theorem proved in~\cite{DNGL}

\begin{theorem}\label{t2}
Let $\Omega$ be bounded hyperconvex domain and let $p>0$. Then for any $0<\gamma<\frac {2n(n+1)}{n+p}$ there exists constant $C_{\gamma}$ such that for all $u\in \mathcal E_p$ holds
\[
\int_{\Omega}\exp\left(\gamma e_p(u)^{-\frac 1n}(-u)^{\frac {n+p}{n}}\right)dV_{2n}\leq C_{\gamma}.
\]
\end{theorem}

\bigskip

\noindent {\bf Theorem~B.} \emph{Let $\Omega$ be a bounded hyperconvex domain in  $\mathbb{C}^n$, $n\geq 2$ and let $0 < p < n$. Then
\[
\Ent_p \subset \mathcal{E}_{\frac{np}{n-p}}.
\]
Furthermore, for $A > 0$, there exist constants $c, C > 0$ depending only on $n, p, A$ such that for any $u \in \Ent_p^A$,
\[
\int_{\Omega} \exp\left(c(n,p,A)(-u)^{\frac{n}{n-p}}\right) dV_{2n} < C(n,p,A).
 \quad \text{and} \quad
 e_{\frac{np}{n-p}}(u) < C(n,p,A)
\]
For $p \geq n$, it holds that $\Ent_p \subset \mathcal{E}_{q}$ for all $q > 0$.
}

\begin{proof} Consider $u \in \Ent_p$ such that $(dd^c u)^n = f dV_{2n}$ and $\int_{\Omega} \varphi_p(f) dV_{2n} < \infty$. Define $f_j = \min(j,f)$. According to Example~\ref{ex}, we have
\[
\int_{\Omega}\Phi_p(f_j) dV_{2n} < \infty, \quad \text{and} \quad \int_{\Omega}\Phi_p(f) dV_{2n} < \infty.
\]
Moreover, we have $f_j dV_{2n} \leq (dd^c j^{\frac{1}{n}}v)^n$, where $v \in \mathcal{E}_0$ satisfies $(dd^c v)^n = dV_{2n}$. By Kołodziej's celebrated subsolution theorem (see~\cite{Ko05}), there exists a $u_j \in \mathcal{E}_0$ such that $(dd^c u_j)^n = f_j dV_{2n}$. The sequence $u_j$ decreases and converges pointwise to $u$ as $j \to \infty$.

Let $q, Q > 0$. Applying the Young inequality (Theorem~\ref{yi}) for $t = f_j$ and $s = \gamma e_q(u_j)^{-\frac{p}{n}}(-u_j)^{Q}$, and referencing (\ref{e0}), we obtain
\begin{multline}\label{1}
\int_{\Omega} \gamma e_q(u_j)^{-\frac{p}{n}}(-u_j)^{Q} f_j dV_{2n} \\
\leq \int_{\Omega} \Phi_p(f_j) dV_{2n} + \int_{\Omega} \Phi_p^*\left(\gamma e_q(u_j)^{-\frac{p}{n}}(-u_j)^{Q}\right) dV_{2n} \\
\leq \int_{\Omega} \Phi_p(f) dV_{2n} + \int_{\Omega} \exp\left(c_p \gamma^{\frac{1}{p}} e_q(u_j)^{-\frac{1}{n}}(-u_j)^{\frac{Q}{p}}\right) dV_{2n} < D,
\end{multline}
where $D$ is a constant.

Setting $Q = q = \frac{np}{n-p}$ and choosing $\gamma > 0$ such that $c_p \gamma^{\frac{1}{p}} < \frac{2n(n+1)}{n+p}$, we deduce from Theorem~\ref{t2} that the right-hand side of (\ref{1}) is finite and the constant $D$ dependent only on the $p$-entropy of $f$ and is independent of $u_j$. Consequently, the left-hand side of (\ref{1}) is also finite and equals
\begin{equation}\label{new}
\gamma \sup_j e_{\frac{np}{n-p}}(u_j)^{\frac{n-p}{n}} = \sup_j \int_{\Omega} \gamma e_q(u_j)^{-\frac{p}{n}}(-u_j)^{Q} f_j dV_{2n} < D < \infty,
\end{equation}
implying that $u \in \mathcal{E}_{\frac{np}{n-p}}$.

If we now assume that $u\in \Ent_p^A$, for some $A>0$, then passing to the limit in (\ref{new}) we obtain
\[
e_{\frac{np}{n-p}}(u)\leq (D\gamma^{-1})^{\frac{n}{n-p}}=C,
\]
where $C>0$ depends only on $n,p,A$ and $\Omega$.

Taking $c=\gamma C^{-\frac 1n}$ and using Theorem~\ref{t2} we also obtain
\[
\begin{aligned}
&\int_{\Omega}\exp\left(c(-u)^{\frac{n}{n-p}}\right)dV_{2n}=
\int_{\Omega}\exp\left(\gamma C^{-\frac 1n}(-u)^{\frac{n}{n-p}}\right)dV_{2n}\\
&\leq\int_{\Omega}\exp\left(\gamma e_{\frac{np}{n-p}}(u)^{-\frac{1}{n}}(-u)^{\frac{n}{n-p}}\right)dV_{2n}\leq C_{\gamma}.
\end{aligned}
\]
To conclude, note that if $p = n$, for any $q > 0$, we can define $p' = \frac{nq}{n+q} < n$. Thus, $q = \frac{np'}{n-p'}$ and consequently
\[
\text{\textbf{Ent}}_{n} \subset \text{\textbf{Ent}}_{p'} \subset \mathcal{E}_{\frac{np'}{n-p'}} = \mathcal{E}_{q},
\]
which completes the proof.
\end{proof}

Now, we provide an example~\cite{C98}, a function $u_\alpha$, such that for a fixed $p > 0$ and $0<\alpha < \frac{n-p}{n}$, we have
that $u_{\alpha} \in \text{\textbf{Ent}}_{p}$ if and only if $u_{\alpha} \in \mathcal{E}_{\frac{np}{n-p}}.$ This confirms the sharpness of the result in Theorem~B. Notably, $u_\alpha$ is unbounded.

\begin{example}
Let $u_{\alpha}(z) = -(\log |z|)^{\alpha} + (\log 2)^{\alpha}$ for $0 < \alpha < 1$, be a plurisubharmonic function in the ball $B(0,\frac{1}{2}) \subset \mathbb{C}^n$. Then $(dd^c u_{\alpha})^n = c f_{\alpha}(z) dV_{2n}$, where $c$ is a constant, and the density function $f_{\alpha}(z)$ equals
\[
f_{\alpha}(z) = \frac{(-\log |z|)^{n(\alpha-1)-1}}{|z|^{2n}}.
\]
Note that $u_{\alpha} \in \mathcal{E}_p$ if and only if $\alpha < \frac{n}{n+p}$. Moreover,
\[
\int_{B(0,\frac{1}{2})} f_{\alpha}(\log f_{\alpha})^q dV_{2n} \approx \int_0^{\frac{1}{2}} \frac{(-\log r)^{n(\alpha-1)-1+q}}{r} dr.
\]
Hence, $f_{\alpha} \in L^{\varphi_q}$ if and only if $\alpha < \frac{n-q}{n}$. Furthermore, $u_{\alpha}$ is an unbounded plurisubharmonic function such that, for fixed $p > 0$ and $\alpha < \frac{n-p}{n}$,
\[
f_{\alpha} \in L^{\varphi_p} \iff u_{\alpha} \in \mathcal{E}_{\frac{np}{n-p}}.
\]\hfill{$\Box$}
\end{example}

Now, we will examine the case when $p = n$. We know that for $p > n$, every $u \in \textbf{Ent}_p$ is bounded, and from Theorem~B, that for $p = n$, $u$ belongs to all Cegrell classes $\mathcal{E}_q$, $q > 0$. To complete the entire picture, we need to address the question of the boundedness of $u$. Finally, we provide an example to demonstrate that for $p = n$, a function $u \in \textbf{Ent}_n$ can indeed be unbounded.

\begin{example}
Let $0 < \beta < \frac{n-1}{n}$ and $u_{\beta}(z) = (\log(-\log |z|))^{\beta} - 1$. Then $u_{\beta}$ is an unbounded plurisubharmonic function in the ball $B(0,e^{-e}) \subset \mathbb{C}^n$. We have $(dd^c u_{\beta})^n = c f_{\beta}(z) dV_{2n}$, where $c$ is a constant and
\[
f_{\beta}(z) = \frac{\log((- \log |z|))^{n(\beta-1)-1} (\log(-\log |z|) + \beta - 1)}{(-\log |z|)^{n+1} |z|^{2n}}.
\]
It is observed that $u_{\beta} \in \mathcal{E}_p$ for all $p > 0$. Additionally,
\[
\int_{B(0,e^{-e})} f_{\beta}(\log f_{\beta})^n dV_{2n} \approx \int_0^{e^{-e}} \frac{\log((- \log r))^{n(\beta-1)}}{(-\log r)r} dr < \infty,
\]
by the assumption that $0 < \beta < \frac{n-1}{n}$. Thus, $f_{\beta} \in L^{\varphi_n}$. \hfill{$\Box$}
\end{example}

From now on, through the rest of this section, we assume that $(X, \omega)$ is a connected and compact K\"ahler manifold of complex dimension $n$, with $n \geq 2$ and $p > 0$. Here, $\omega$ denotes a K\"ahler form on $X$ satisfying $\int_{X} \omega^n = 1$.

For any function $u$ in the space of $\omega$-plurisubharmonic functions, $\mathcal{PSH}(X, \omega)$, we define
\[
\omega_u = dd^c u + \omega.
\]

We consider $\mathcal{E}(X, \omega)$ to be the class of all $\omega$-plurisubharmonic functions, described as
\[
\mathcal{E}(X, \omega) = \left\{ u \in \mathcal{PSH}(X, \omega) : \int_X \omega_u^n = 1 \right\}.
\]
We then define the class of \emph{$\omega$-plurisubharmonic functions with bounded $p$-energy}, $\mathcal{E}_p(X, \omega)$, as
\begin{equation}\label{Def Ep}
\mathcal{E}_p(X, \omega) := \left\{ u \in \mathcal{E}(X, \omega) : u \leq 0, e_p(u) = \int_X (-u)^p \omega_u^n < \infty \right\}.
\end{equation}

We shall refer to the following theorem proven in~\cite{DNGL}  in the proof of Theorem~A.

\begin{theorem}\label{man}
Let $p > 0$. Then there exist constants $c, C > 0$ dependent only on $X, \omega, n,$ and $p$, such that for all $u \in \mathcal{E}_p(X, \omega)$ with $\sup_X u = -1$, the following inequality holds:
\[
\int_{X} \exp\left(c e_p(u)^{-\frac{1}{n}}(-u)^{\frac{n+p}{n}}\right) \omega^n \leq C.
\]
\end{theorem}

Utilizing the proof of Theorem~B and Theorem~\ref{man}, we arrive at the main result of this paper:

\bigskip

\noindent {\bf Theorem A.} \emph{Let $u$ be an $\omega$-plurisubharmonic function
on a connected and compact K\"ahler manifold $(X,\omega)$, with $\omega$ being  a normalized K\"ahler form on $X$ of complex dimension $n \geq 2$, and $\sup_X u=-1$. Let $0<p<n$. If the probability measure $\mu = (dd^c u + \omega)^n = f\omega^n$ has finite $p$-entropy, i.e.,
\begin{equation}\label{Pentr}
\int_X (1+f)(\log (1+f))^p \omega^n < \infty,
\end{equation}
then $u$ belongs to $\mathcal{E}_{\frac{np}{n-p}}$. If additionally,
\begin{equation*}
\int_X (1+f)(\log (1+f))^p \omega^n < A,
\end{equation*}
for a constant $A > 0$, then there are constants $c, C > 0$ depending only on $n, p, A$ such that any probability measure $\mu = (dd^c u + \omega)^n = f \omega^n$ with uniform bound of its $p$-entropy in the sense of~\eqref{Pentr_II} satisfies:
\[
\int_X \exp\left(c(n,p,A)(-u)^{\frac{n}{n-p}}\right) \omega^n < C(n,p,A)\quad \text{and} \quad e_{\frac{np}{n-p}}(u) < C(n,p,A).
\]}

\section{Stability}\label{Sec_Stability}

In Theorem~A and Theorem~B, we established that $\Ent_p \subset \mathcal{E}_{\frac{np}{n-p}}$, demonstrating the existence of a unique solution $u \in \mathcal{E}_{\frac{np}{n-p}}$ for the Dirichlet problem associated with the complex Monge-Amp\`ere equations:
\[
(dd^cu+\omega)^n=f\omega^n, \quad \text{and} \quad (dd^c u)^n = f\,dV_{2n},
\]
where the density function $f$ is a member of the Orlicz space $L^{\Phi_p}$. Following the verification of a solution's existence, the next logical step is to examine the stability of this solution.

Stability in the context of the complex Monge-Amp\`ere equation is defined as follows: slight perturbations of the right-hand side function $g$ within the space $L^{\Phi_p}$ should result in a new solution that is close to the original in a specific sense. In the literature, strong stability results, where the supremum norm of continuous or bounded solutions is estimated by the $L^p$ norm of density differences, have been well documented (see~\cite{Czy09}).

For unbounded solutions, where uniform convergence is unattainable, the goal may shift to achieving a weak stability, characterized by convergence in capacity or weak convergence (in the $L^q$ norm) of the solutions.

Since the prove of stability results are different for hyperconvex domains and compact K\"ahler manifolds we shall provide both of them.

\subsection{Stability in hyperconvex domains}

This section aims to establish a stability theorem for the complex Monge-Amp\`ere equation within the class $\Ent_p$ for $p \leq n$. Given that our solutions may be unbounded, we seek to demonstrate a form of convergence that surpasses both weak convergence and convergence in capacity. Specifically, we will prove convergence in the quasi-metric $\Jp$, defined on $\mathcal{E}_{p}(\Omega)$ as follows.

For $u,v\in \mathcal E_{p}(\Omega)$ and $p>0$ let us define
\[
\Jp(u,v)=\left(\int_{\Omega}|u-v|^p((dd^cu)^n+(dd^cv)^n)\right)^\frac {1}{p+n}.
\]
It was proved in~\cite{AC22} that $(\mathcal E_{p}(\Omega),\Jp)$ is a complete quasi-metric space.

We shall need the following comparison principle, see~\cite{ACCP09,NP09}.

\begin{theorem}\label{cp}
Let $\Omega$ be a bounded hyperconvex domain, let $u,v\in \mathcal F(\Omega)$, $w\in \mathcal E_{0}(\Omega)$ then
\[
\int_{\{u<v\}}(v-u)^n(dd^cw)^n\leq n!\|w\|_{L^\infty}^{n-1}\int_{\{u<v\}}(-w)((dd^cu)^n-(dd^cv)^n).
\]
\end{theorem}

The first goal of this section is to prove the following stability theorem.

\begin{theorem}\label{thm_stability}
Let $n \geq 2$, $0< p < n$, and define $p_0 = \frac{np}{n-p}$. Assume that $\Omega \subset \mathbb{C}^n$ is a hyperconvex domain and let $\int_{\Omega}\varphi_p(g)dV_{2n}<\infty$. If $0 \leq f, f_j \leq 1$ are measurable functions such that $f_jg \to fg$ in $L^{\Phi_p}$ as $j \to \infty$, then $\mathbf{J}_q(u_j,u) \to 0$ for $q < p_0$, where $u_j,u \in \mathcal{E}_{p_0}$ are the unique functions satisfying
\[
(dd^c u_j)^n = f_jg dV_{2n} \quad \text{and} \quad (dd^c u)^n = f g dV_{2n}.
\]
\end{theorem}

\begin{remark}
It is important to note, as proven in~\cite{DNGL}, that for $p = 1$, the convergence $\mathbf{J}_q(u_j,u) \to 0$ in Theorem~\ref{thm_stability} does not hold when $q = \frac{n}{n-1}$.
\end{remark}

\begin{proof}
First, by the proof of Theorem~\ref{man}, we have
\begin{equation}\label{e3}
\sup_j (e_{p_0}(u_j), e_{p_0}(u)) \leq C(g),
\end{equation}
where $C(g)$ depends on the $p_0$-entropy of $g$ and is independent of $j$. Furthermore, by the Sobolev type inequality (see~\cite{AC19}), for any $\alpha > 0$,
\begin{equation}\label{e4}
\int_{\Omega} |u_j - u|^{\alpha} dV_{2n} \leq C \left(\sup_j (e_{p_0}(u_j), e_{p_0}(u))\right)^{\frac {\alpha}{n+p_0}} \leq C_1(g),
\end{equation}
where $C_1(g)$ depends on the $p_0$-entropy of $g$ and is independent of $j$. Fix $q < p_0$ and $M > 0$. We have
\begin{equation}\label{es1}
\begin{aligned}
\mathbf{J}_q(u_j,u)^{q+n} & = \int_{\Omega} |u_j - u|^q (f_jg + fg) dV_{2n}
\leq 2 \int_{\Omega} |u_j - u|^q g dV_{2n} \\
&\leq 2 \left(\int_{\Omega \cap \{g > M\}} |u_j - u|^q g dV_{2n} + \int_{\Omega \cap \{g \leq M\}} |u_j - u|^q g dV_{2n}\right) \\
&= 2(I_1 + I_2).
\end{aligned}
\end{equation}
Estimate the integral $I_1$ using H\"older's inequality:
\begin{equation}\label{es2}
\begin{aligned}
I_1 &= \int_{\Omega \cap \{g > M\}} |u_j - u|^q g dV_{2n} \\
&\leq \left(\int_{\Omega \cap \{g > M\}} |u_j - u|^{p_0} g dV_{2n}\right)^{\frac{q}{p_0}} \left(\int_{\Omega \cap \{g > M\}} g dV_{2n}\right)^{\frac{p_0 - q}{p_0}} \\
& = \left(I_3\right)^{\frac{q}{p_0}} \left(I_4\right)^{\frac{p_0 - q}{p_0}}.
\end{aligned}
\end{equation}

Let $\psi_0 \in \mathcal{E}_{p_0}$ be such that $(dd^c \psi_0)^n = g dV_{2n}$. Then by (\ref{e3}) and Theorem~\ref{thm_holder},
\begin{equation}\label{es3}
\begin{aligned}
I_3 & = \int_{\Omega \cap \{g > M\}} |u_j - u|^{p_0} (dd^c \psi_0)^n \\
& \leq \int_{\Omega} (-u_j - u)^{p_0} (dd^c \psi_0)^n \\
& \leq d \sup_j (e_{p_0}(u_j)^{\frac{p_0}{p_0+n}}, e_{p_0}(u)^{\frac{p_0}{p_0+n}}, e_{p_0}(\psi_0)^{\frac{n}{p_0+n}}) \leq C_2(g),
\end{aligned}
\end{equation}
where $C_2(g)$ depends on the $p_0$-entropy of $g$ and is independent of $j$.

The integral $I_4$ can be estimated as follows:
\begin{equation}\label{es4}
I_4 = \int_{\Omega \cap \{g > M\}} g dV_{2n} \leq \frac{C_3(g)}{(\log(1 + M))^{p}},
\end{equation}
where $C_3(g)$ depends on the $p_0$-entropy of $g$ and is independent of $j$.

Combining (\ref{es2}), (\ref{es3}), and (\ref{es4}), we obtain:
\begin{equation}\label{f1}
I_1 \leq C_4(g) (\log(1 + M))^{\frac{p(q - p_0)}{p_0}},
\end{equation}
where $C_4(g)$ depends on the $p_0$-entropy of $g$, and is independent of $j$.

Now, we estimate the integral $I_2$. Let $w \in \mathcal{E}_0$ be such that $(dd^c w)^n = dV_{2n}$. Then,
\begin{equation}\label{es5}
I_2 = \int_{\Omega \cap \{g \leq M\}} |u_j - u|^q g dV_{2n} \leq M \int_{\Omega} |u_j - u|^{q} (dd^c w)^n.
\end{equation}

\medskip

\noindent Consider two cases:

\medskip

\noindent\emph{Case 1.} Assume first that $q \leq n$. Then, using Theorem~\ref{orlicz_holder} and Theorem~\ref{cp}
\begin{equation}\label{es6}
\begin{aligned}
I_2 & \leq M \left(\int_{\Omega} |u_j - u|^{n} (dd^c w)^n\right)^{\frac{q}{n}} V_{2n}(\Omega)^{\frac{n - q}{n}} \\
& \leq M\, V_{2n}(\Omega)^{\frac{n - q}{n}} \left(n! \|w\|_{L^{\infty}}^{n-1} \int_{\Omega} (-w) |f_jg - fg| dV_{2n}\right)^{\frac{q}{n}} \\
& \leq M \, V_{2n}(\Omega)^{\frac{n - q}{n}}\left(n! \|w\|_{L^{\infty}}^{n-1}\right)^{\frac qn} \|f_jg - fg\|_{\Phi_p}^{\frac{q}{n}}(\|w\|_{\Phi_p^*}^0)^{\frac{q}{n}}, \\
& \leq M C_5\, \|f_jg - fg\|_{\Phi_p}^{\frac{q}{n}},
\end{aligned}
\end{equation}
where $C_5$ is a constant independent of $j$.

\medskip

\noindent\emph{Case 2.} Now assume that $q > n$. Using Theorem~\ref{orlicz_holder}, (\ref{e4}), and (\ref{es6}), we obtain:
\begin{multline}\label{es7}
I_2 \leq M \left(\int_{\Omega} |u_j - u|^{n} (dd^c w)^n\right)^{\frac{n - 1}{n}} \left(\int_{\Omega} |u_j - u|^{n(q - n + 1)} (dd^c w)^n\right)^{\frac{1}{n}} \\
\leq M C_6(g) \|f_jg - fg\|_{\Phi_p}^{\frac{n - 1}{n}},
\end{multline}
where $C_6(g)$ depends on the $p_0$-entropy of $g$, and is independent of $j$.

\medskip

In both cases (\ref{es6}) and (\ref{es7}), we have:
\begin{equation}\label{es8}
I_2 \leq M C_7(g) \|f_jg - fg\|_{\Phi_p},
\end{equation}
where $C_7(g)$ depends on the $p_0$-entropy of $g$, and is independent of $j$. Finally, to prove that $\mathbf{J}_q(u_j,u)^{q+n} \to 0$ as $j \to \infty$, fix $\epsilon> 0$. Then choose $M > 0$ large enough such that $I_1 \leq \frac{\epsilon}{4}$, see (\ref{f1}). Since $\|f_jg - fg\|_{\Phi_p} \to 0$ by assumption, we can choose $j$ large enough so that $I_2 \leq \frac{\epsilon}{4}$, according to (\ref{es8}). This leads to
\[
\mathbf{J}_q(u_j,u)^{q+n} = 2(I_1 + I_2) \leq 2\left(\frac{\epsilon}{4} + \frac{\epsilon}{4}\right) = \epsilon,
\]
which implies $\mathbf{J}_q(u_j,u) \to 0$ as $j \to \infty$, thereby concluding the proof.
\end{proof}

\subsection{Stability on compact K\"ahler manifolds}
Recall that we assume $n \geq 2$ and $p > 0$. Consider $(X, \omega)$ a connected and compact K\"ahler manifold of complex dimension $n$, where $\omega$ is a K\"ahler form on $X$ normalized such that $\int_{X} \omega^n = 1$.

Thanks to the work of Darvas~\cite{D1, D2}, we have a Finsler metric ${\bf d}_p$ on $\mathcal{E}_p(X, \omega)$ making $(\mathcal{E}_p(X, \omega), {\bf d}_p)$ into a complete geodesic metric space.  In Theorem~\ref{thm_stability_cKm} we need that for any $u, v \in \mathcal{E}_p(X, \omega)$ and for $p > 0$, there exists a constant $C = C(n, p)$, which depends only on $p$ and $n$, such that:
\begin{equation}\label{ineqdp}
C^{-1} \int_X |u - v|^p (\omega_u^n + \omega_v^n) \leq {\bf d}_p(u, v) \leq C \int_X |u - v|^p (\omega_u^n + \omega_v^n).
\end{equation}

Next, we present our stability result.

\begin{theorem}\label{thm_stability_cKm}
Let $n \geq 2$, $0< p < n$, and define $p_0 = \frac{np}{n-p}$. Assume that $g$ is measurable function such that $\int_{\Omega}\varphi_p(g)\omega^n<\infty$. If $0 \leq f, f_j \leq 1$ are measurable functions such that $\int_Xf_jg\omega^n=\int_Xfg\omega^n=1$ and $f_jg \to fg$ in $L^{\Phi_p}(X)$ as $j \to \infty$, then $\mathbf{d}_q(u_j,u) \to 0$ for $q < p_0$, where $u_j,u \in \mathcal{E}_{p_0}(X,\omega)$ are the unique functions satisfying $\sup_Xu_j=\sup_Xu=-1$ and
\[
\omega_{u_j}^n = f_jg \omega^n \quad \text{and} \quad \omega_u^n = f g \omega^n.
\]
\end{theorem}

\begin{proof}
First, observe that the convergence $f_jg \to fg$ in $L^{\Phi_p}(X)$ as $j \to \infty$ implies
\[
\int_X|f_jg-fg|\omega^n\to 0 \quad \text{as } j\to \infty.
\]
Therefore, by Proposition 12.17 in~\cite{GZbook}, $u_j\to u$ weakly (e.g., in $L^r(X,\omega)$ for all $r>0$).

Using the same argument as in the proof of Theorem~\ref{thm_stability}, we establish that
\[
\sup_j (e_{p_0}(u_j), e_{p_0}(u)) \leq C(g),
\]
where $C(g)$ depends on the $p_0$-entropy of $g$ and is independent of $j$.

Let us denote by $h_j = f_jg + fg$ and fix $q < p_0$. To prove that ${\bf d}_q(u_j, u) \to 0$ by (\ref{ineqdp}), it is sufficient to demonstrate that
\begin{equation}\label{int1}
\int_X|u-u_j|^q(\omega_u^n+\omega_{u_j}^n)\to 0 \quad \text{as } j\to \infty.
\end{equation}
By Theorem~\ref{orlicz_holder}, we have
\[
\int_X|u_j-u|^qh_j\omega^n \leq \|h_j\|_{L^{\Phi_p}}\||u_j-u|^q\|^0_{L^{\Phi^*_p}}.
\]
Since $\|h_j\|_{L^{\Phi_p}}$ is uniformly bounded, it is enough to show that
\[
\||u_j-u|^q\|^0_{L^{\Phi^*_p}} \to 0 \quad \text{as } j\to \infty.
\]
To establish this, we shall show that there exists constant $\lambda > 0$ such that
\[
\int_X\Phi_p^*\left(\lambda|u_j-u|^q\right)\omega^n \to 0 \quad \text{as } j\to \infty.
\]
Recall that since $u_j, u \in \mathcal E_{p_0}(X,\omega)$, by Theorem~A, there exist constants $c > 0$ and $C > 0$ independent of $j$ such that
\begin{equation}\label{e10}
\int_X \exp\left(c(-u_j-u)^{\frac{n}{n-p}}\right) \omega^n \leq C.
\end{equation}
Noting that $\frac{q}{p} < \frac{n}{n-p}$ and using (\ref{e0}) and (\ref{e10}), we get for fixed $\lambda=(\frac{c}{q+1})^p$
\begin{multline*}
\int_X\Phi_p^*\left(\lambda|u_j-u|^q\right)\omega^n \leq \lambda\int_X\exp\left(\lambda^{\frac{1}{p}}|u_j-u|^{\frac{q}{p}}\right)|u_j-u|^q\omega^n \\
\leq\lambda\int_X\exp\left(\lambda^{\frac{1}{p}}(-u_j-u)^{\frac{q}{p}}\right)|u_j-u|^q\omega^n\\
\leq \lambda \left(\int_X\exp\left((q+1)\lambda^{\frac{1}{p}}(-u_j-u)^{\frac{q}{p}}\right)\omega^n\right)^{\frac{1}{q+1}}\left(\int_X|u_j-u|^{1+q}\omega^n\right)^{\frac{q}{q+1}} \\
\leq \lambda C\left(\int_X|u_j-u|^{1+q}\omega^n\right)^{\frac{q}{q+1}} \to 0 \quad \text{as } j\to \infty.
\end{multline*}
\end{proof}

\end{document}